\documentclass[12pt,a4paper]{article}

\usepackage[a4paper, top=2.5cm, bottom=2.5cm, left=2.5cm, right=2.5cm]{geometry}

\usepackage[utf8]{inputenc}
\usepackage[T1]{fontenc}

\usepackage{amsmath,amssymb,amsthm}
\usepackage{bm}

\usepackage{enumitem}

\usepackage{graphicx}
\usepackage{subcaption}
\graphicspath{{./figures/}}

\usepackage{ulem}
\usepackage{xcolor}
\newcommand\rso{\bgroup\markoverwith{\textcolor{red}{\rule[0.5ex]{2pt}{1pt}}}\ULon}

\usepackage{authblk}

\usepackage{hyperref}
\hypersetup{
    colorlinks=true,
    linkcolor=blue,
    filecolor=magenta,
    urlcolor=cyan,
}

\newcommand{\R}{\mathbb R}

\newcommand{\E}{\mathbf{E}}
\newcommand{\Var}{\mathbf{Var}}
\newcommand{\pr}{\mathbb P}
\newcommand{\N}{\mathbb N}

\theoremstyle{definition}
\newtheorem{definition}{Definition}[section]
\newtheorem{remark}[definition]{Remark}

\theoremstyle{plain}
\newtheorem{theorem}{Theorem}[section]

\newtheorem{corollary}[theorem]{Corollary}
\newtheorem{proposition}[theorem]{Proposition}

\newtheorem{assumption}{Assumption}

\usepackage{xifthen}
\newcommand{\textcite}[2][]{\ifthenelse{\isempty{#1}}{\cite{#2}}{\cite[#1]{#2}}}
\newcommand{\parencite}[2][]{\ifthenelse{\isempty{#1}}{\cite{#2}}{\cite[#1]{#2}}}

\title{Asymptotic Normality and Convergence Rates for Tsallis Entropy Estimators via Stabilization Techniques }
\author[1]{\footnotesize Mehmet S\i dd\i k \c{C}ad\i rc\i}
\author[2*]{\footnotesize Martin Singull}
\affil[1]{\footnotesize Faculty of Science, Department of Statistics, Cumhuriyet University, Sivas, T\"urkiye}
\affil[2]{\footnotesize Department of Mathematics, Link\"oping University, Link\"oping, Sweden}
\affil[*]{\footnotesize Corresponding author: \href{mailto:martin.singull@liu.se}{\texttt{martin.singull@liu.se}}}
\date{}

\begin{document}

\maketitle

\begin{abstract}
We study nearest-neighbor-based estimators of Tsallis entropy associated with Poisson and binomial point processes on general metric measure spaces. Using stabilization techniques based on flexible add-one cost operators together with second-order Poincaré inequalities, we establish asymptotic normality and derive explicit convergence rates  for the Kolmogorov distance. Our analysis avoids explicit score-function decompositions and instead relies on flexible localizations of add-one costs, which simplify the treatment of higher-order terms. Under natural stabilization and moment conditions, the resulting bounds recover the classical normal approximation rates \(s^{-1/2}\) and \(n^{-1/2}\) and extend corresponding results for Shannon and Rényi entropy estimators. We further illustrate the scope of the framework through examples involving Tsallis entropy functionals, weighted \(k\)-nearest-neighbor Shannon entropy estimators. The examples provided highlight the benefits of stabilization-based normal approximations for non-parametric statistical inference in complex spatial and high-dimensional settings.
\end{abstract}

\section{Introduction}
\label{sec:intro}

Suppose that \(\mathcal{X}, \mathcal{F}, d, Q\) form a metric measure space with a \(\sigma\)-finite measure \(Q\), and a metric \(d\colon \mathcal{X}\times\mathcal{X}\to[0,\infty)\). Let \(P_s\) denote a Poisson point process on \(\mathcal{X}\) having a density measure \(\lambda:=sQ\) for \(s\ge 1\). Given that \(Q\) has a probability measure, then we consider the corresponding binomial point process \(\xi_n\) composed of \(n\) i.i.d. observed values derived from \(\lambda = sQ\). We are concerned in this paper with the asymptotic distributional  properties of Tsallis entropy estimators derived from the nearest-neighbor distances of these point processes and with the derivation of sharp quantitative normal approximation bounds for such estimators.

Entropy and associated information measures play a central role in statistics, information theory, machine learning, and the analysis of complex systems \cite{cover2006elements,jaynes1957information}. This paper aims to develop a new framework for studying non-extensive entropy measures. Many systems of practical interest, such as those encountered in finance, hydrology, turbulence, and networked dynamical systems exhibit non-extensive behavior that is poorly captured by purely logarithmic entropy measures. Tsallis entropy \cite{tsallis1988possible} is a parametric family of generalized entropies that lies between heavy-tailed and compactly supported regimes and has become a standard tool for studying non-equilibrium and long-range-dependent phenomena.

Methodologically, nearest  neighbor-based estimators form a flexible class of non-parametric entropy estimators that do not involve direct density estimation. Regarding Shannon entropy, this series of studies includes the Kozachenko-Leonenko estimator and its improvements \cite{berrett2019efficient}, while for generalized entropies such as Rényi and Tsallis, several \(k\)-nearest neighbor (\(k\)-NN) structures are proposed and evaluated empirically. Specifically, Tsallis-based goodness-of-fit tests using \(k\)-NN estimators are developed in \cite{cadirci2025tsallisGOF}; there, the emphasis is on test statistics for multivariate generalized Gaussian and \(q\)-Gaussian distributions. Entropy-based tests for generalized Gaussian distributions based on Shannon entropy and maximum entropy principles have been investigated in \cite{cadirci2022ggd}. Supplementary Rényi-based goodness-of-fit tests are proposed in \cite{cadirci2025renyi} for multivariate Student and Pearson type II distributions, based on the principles of maximum Rényi entropy and nearest neighbor Rényi entropy estimators.

Although these studies highlight the practical potential of Shannon, Tsallis, and Rényi entropy estimators for goodness-of-fit testing, their focus is primarily on consistency, mean-square convergence, and calibration of critical values via Monte Carlo methods. Consistency for a class of Rényi entropy estimators and examines the empirical behaviour of the corresponding test statistics. At the same time, \cite{cadirci2022ggd,cadirci2025tsallisGOF} presents extensive simulation results for entropy-based tests under generalized Gaussian and related models.

The complex local dependency structure of nearest neighbor-based functionals creates a significant technical challenge by complicating the application of standard central limit theorems. Introduced in \cite{kesten1996central} for geometric functionals of random point sets and elaborated further in \cite{penrose2001central,penrose2005multivariate,baryshnikov2005gaussian,chatterjee2017minimal} , further developed in \cite{penrose2001central,penrose2005multivariate,baryshnikov2005gaussian,chatterjee2017minimal}, provide a powerful framework for capturing the local dependence. In more recent progress, combining stabilization ideas with Malliavin-Stein methods and second-order Poincaré inequalities leads to the derivation of quantitative approximation bounds for a wide class of Poisson and binomial point processes \cite {last2016normal,lachieze2019normal,lachieze2017new,lachieze2022stabilization,shi2024flexible}. Together, these advances have led to sharp central limit theorems for geometric statistics, such as the volumes, face counts, and Betti numbers of random complexes, as well as for functionals related to random graphs.

\textbf{Contributions.} We summarise the main contributions of the paper as follows. First, instead of explicit point separations, we propose a flexible stabilization framework for Tsallis entropy functions on Poisson and binomial point processes based on adaptive addition cost operators. This approach builds upon and extends the cost operator methodology in \cite{lachieze2022stabilization} and develops the normal approximation results in \cite{lachieze2019normal,last2016normal}. It is designed to accommodate functionals that are naturally expressed in terms of nearest neighbor distances.

Second, we establish central limit theorems and rates of convergence in the Kolmogorov distance for the general stabilization functionals of Poisson and binomial point processes under a series of transparency, tail, and moment conditions on the fundamental metric measure space and stabilization radii. We express the resulting bounds in terms of integrals involving auxiliary quantities that capture the first- and second-order approximation costs and the exponential decay of stabilization radii. Additionally, we present a simple sufficient condition, represented by quantities \(\Theta_{K,s}\) and \(\Theta_{K,n}\), under which the Kolmogorov distance reaches its optimal order \(s^{-1/2}\) or \(n^{-1/2}\).

Thirdly, we customize this general theory for Tsallis entropy estimators based on \(k\)-NN distances, deriving asymptotic normality of order \(s^{-1/2}\) and \(n^{-1/2}\) for Poisson and binomial inputs, respectively. In the classical sense of the central limit theorem, we show that these limits are optimal up to constants and rigorously generalize the known normal approximation results for Shannon and Rényi entropy estimators within the same geometric framework. We complement the Rényi-based estimation and testing results in \cite{cadirci2025renyi} by providing a stabilization-based path to Gaussian approximations for the relevant functionals.

Fourth, we discuss other applications of weighted \(k\)-NN Shannon entropy estimation, in line with the spirit of \cite{berrett2019efficient}, Euler characteristic functionals of random geometric complexes, and minimum spanning tree (MST) statistics, to demonstrate the versatility of the approach. The framework reproduced or improved available rates for each of these examples by avoiding explicit stabilization radius estimations and score function decompositions.

\textbf{Relation to goodness-of-fit tests.} The current results complete the entropy-based goodness-of-fit tests for generalized Gaussian and associated models. In the study \cite{cadirci2022ggd}, the authors developed \(k\)-NN Shannon entropy estimators and tests derived from maximum entropy principles based on generalized Gaussian families, obtained the \(\text{L}^2\)-consistency of the entropy estimator, and performed extensive simulations on both empirical dimension and power. In \cite{cadirci2025tsallisGOF}, Tsallis entropy is employed to generate goodness-of-fit statistics for multivariate generalized Gaussian and \(q\)-Gaussian models, revealing that the Tsallis-based estimator performs effectively in scenarios involving heavy tails and non-Gaussian distributions. The Rényi-based procedures in \cite{cadirci2025renyi} employ nearest-neighbor Rényi entropy estimators and maximum entropy characterizations for multivariate Student and Pearson type II distributions. Our work provides normal approximation results with explicit bounds for Tsallis and associated entropy estimators, providing an abstract probability basis that can be utilized to motivate Gaussian methods and to guide calibration of critical values in such test problems beyond bootstrap or fully simulation-based methods.

The remainder of the paper is organized as follows. Section~\ref{sec:preliminaries} presents the basic concepts of point processes, incremental operators, and stabilization, and formulates the basic assumptions employed throughout the paper. Section~\ref{sec:main_results} provides general normal approximation outcomes for Poisson and binomial inputs, accompanied by a clear bound outcome. Section~\ref{sec:applications} extends these results by applying Tsallis entropy estimators and corresponding geometric statistics; we also summarize a simulation strategy and typical summary statistics. Section~\ref{sec:proofs} provides proofs for the principal theorems and results. Section~\ref{sec:conc} provides a brief overview of potential extensions.

\section{Preliminaries}
\label{sec:preliminaries}

In this section, the fundamental point process framework is recalled, the increment and cost operators are described, and the stabilization and moment assumptions underlying our main results are stated. Keeping the presentation concise, we focus on the elements necessary for subsequent standard approach boundaries.

\subsection{Point processes and functionals}
Suppose \(\mathcal{X}, \mathcal{F}, d, Q\) is a metric measure space with a \(\sigma\)-finite measure \(Q\). Let \(\mathcal{X}\) be equipped with the smallest \(\sigma\)-algebra that makes the following mappings
\begin{equation}
    m_A\colon \mathcal{N}\to \N\cup\{0,\infty\}, \qquad m_A(M)=M(A),
\end{equation}
measurable for every \(A\in\mathcal{F}\). For a point process \(\eta\) on \(\mathcal{X}\), \(\eta\) takes values in \(\mathcal{N}\). We denote the space of \(\eta\)'s square integrable measurable functionals satisfying the condition \(\E[F(\eta)^2]<\infty\) by \(F\colon\mathcal{N}\to\R\) as \(L^2_\eta(\mathcal{X})\).

\begin{definition}[Poisson point process {\cite{kingman1993poisson}}]
A Poisson point process \(P(\lambda)\) on \(\mathcal{X}\) having intensity density \(\lambda\) satisfies the following conditions:
\begin{enumerate}
\item Given any measurable set \(B\in\mathcal{F}\), the random variable \(P(\lambda)(B)\) exhibits a Poisson distribution having parameter \(\lambda(B)\).
\item The random variables \(P(\lambda)(B_1),\dots,P(\lambda)(B_m)\) are independent for pairwise disjoint sets \(B_1,\dots,B_m\in\mathcal{F}\).
\end{enumerate}
\end{definition}

\begin{definition}[Binomial point process {\cite{daley2007introduction}}]
Consider a probability measure \(Q\) on \(\mathcal{X},\mathcal{F}\) and \(n\in\N\). We define the binomial point process \(\xi_n\) as follows:
\begin{equation}
\xi_n = \sum_{i=1}^n \delta_{X_i},
\end{equation}
where \(X_1,\dots,X_n\) follow an i.i.d.  \(Q\) law, and \(\delta_x\) represents the Dirac measure at \(x\in\mathcal{X}\).
\end{definition}
For two real-valued random variables \(Y\) and \(Z\), the Kolmogorov measure is defined as
\begin{equation}
    d_K(Y,Z) = \sup_{t\in\R} \bigl|\pr(Y\le t)-\pr(Z\le t)\bigr|.
\end{equation}
 Moreover, we will measure this distance to assess convergence to the standard normal limit.

\subsection{Increment operators and stabilization}
Methods of stabilization are based on the behaviour of functionals under local modifications to the base-point configuration. Increment operators formalize this idea.

\begin{definition}[Increment operators {\cite{lachieze2019normal}}]
Suppose \(F\colon\mathcal{N}\to\R\) be a measurable functional and \(\eta\) is a point process on \(\mathcal{X}\). For \(x\in\mathcal{X}\), the (first-order) increment operators are as follows
\begin{equation}
\Delta_x F(\eta) := F(\eta\cup\{x\}) - F(\eta).
\end{equation}
For different \(x,y\in\mathcal{X}\), the second-order increment operator can be expressed as
\[
\Delta_{x,y}F(\eta) := F(\eta\cup\{x,y\}) - F(\eta\cup\{x\}) - F(\eta\cup\{y\}) + F(\eta).
\]
\end{definition}

We call a function stable if it becomes insensitive to large changes in the configuration.

\begin{definition}[Score-based stability {\cite{penrose2005multivariate}}]
Consider a measurable score function \(f\colon\mathcal{X}\times\mathcal{N}\to\R\). Then, for \(x\in\mathcal{X}\), \(f\) is said to be stabilized at \(x\) if, for every finite set \(A\subset \mathcal{X}\setminus B_x(R_x)\), there exists an almost surely finite random radius \(R_x>0\) such that
\begin{equation}
    f\bigl(x,(\eta\cap B_x(R_x))\cup A\bigr) = f\bigl(x,\eta\cap B_x(R_x)\bigr),
\end{equation}
\end{definition}
where \( B_x(R_x) \) is the ball of radius \( R_x \) centered at \( x \).
Score-based stabilization can be useful when it can be written as the sum of the local contributions of the functional. For the present work, it is convenient to rely on the strong stabilization properties expressed directly in terms of \(\Delta_x F\) and \(\Delta_{x,y}F\), which are more natural for nearest neighbor-based entropy estimators.

\begin{proposition}
Suppose \(F_s(P_s) = \sum_{x\in P_s} f_s(x,P_s)\) is a functional of a Poisson process \(P_s\) having an intensity of \(sQ\). Then, for \(x\in\mathcal{X}\),
\begin{equation}
\Delta_x F_s(P_s) = f_s(x, P_s\cup\{x\}) + \sum_{y\in P_s} \Delta_x f_s(y,P_s).
\end{equation}
A similar identity stands for functionals of binomial point processes.
\end{proposition}

\begin{remark}\label{remark 2.1}
The strong stabilization of \(F_s\) suggests score-based stabilization for a reasonable choice of scores, however, the reverse may not hold. Working directly with \(\Delta_x F_s\) and \(\Delta_{x,y} F_s\) avoids explicit score decompositions and is often more robust for complex functionals such as Tsallis entropy estimators.
\end{remark}

\subsection{Assumptions}
In this section, the main structural assumptions on the underlying space and the stabilizing functionals considered are stated, adapted from \cite{lachieze2019normal,lachieze2022stabilization,penrose2007gaussian}.

\begin{assumption}[Regularity of \(Q\)]\label{Assumption 2.1}
 We assume the measure \(Q\) on \((\mathcal{X},\mathcal{F})\) satisfies a condition of regular growth: there are constants \(\alpha>0\) and \(\beta>1\) such that, for all \(x\in\mathcal{X}\), all \(r>0\), and sufficiently small \(\epsilon>0\),
\[
\limsup_{\epsilon\to0^+} \frac{Q(B_x(r+\epsilon)) - Q(B_x(r))}{\epsilon}
\le \alpha \beta r^{\beta-1}.
\]
Particularly, \(Q\) has a diffuse property in the sense that \(Q(\{x\})=0\) for all \(x\in\mathcal{X}\); see \cite[Lemma~5.1(a)]{lachieze2019normal}.
\end{assumption}

\begin{assumption}[Tail bound on stabilization radius.]\label{Assumption 2.2}
Assume that \(F_s\) is a strongly stabilizing function of \(P_s\) having a stabilization radius \(R_x\) at the point \(x\). Assume there exist constants \(C_1,C_2,\gamma>0\) with the following inequality for all \(r\ge0\):
\[
\pr(R_x\ge r) \le C_1 \exp\bigl(-C_2 (s^{1/\beta} r)^\gamma\bigr).
\]
For binomial processes \(\xi_n\), a similar bound is assumed using \(n\) instead of \(s\).
\end{assumption}

\begin{assumption}[Exponential bound on add-one costs.]\label{Assumption 2.3}
Suppose \(K\subseteq\mathcal{X}\) is a measurable set with \(F_s\) defined on \(P_s\). Suppose there exist constants \(C_3,C_4,\delta>0\) such that, for all \(u>0\) and \(r\ge0\),
\begin{align*}
\pr(|\Delta_x F_s|\ge u)
&\le C_3 \exp\bigl(-C_4 d_s(x,K)^\delta\bigr),\\
\pr(|\Delta_{x,y} F_s|\ge u)
&\le C_3 \exp\Bigl(-C_4 \max\{d_s(x,y), d_s(x,K), d_s(y,K)\}^\delta\Bigr),
\end{align*}
where \(d_s(x,K) := \inf_{z\in K} s^{1/\beta} d(x,z)\). Similar conditions hold for the binomial case where \(n\) is used instead of \(s\).
\end{assumption}

\begin{assumption}[Uniform moment condition.]\label{Assumption 2.4}
Let there exist \(p>4\) and a definite constant \(M_p\) where
\begin{equation}
\sup_{s\ge 1} \sup_{x,y\in\mathcal{X}} \E\bigl[|\Delta_x F_s|^p + |\Delta_{x,y}F_s|^p\bigr] \le M_p.
\end{equation}
\end{assumption}

\begin{remark}\label{Remark 2.2.}
 These assumptions have been standard in the stabilization literature and are satisfied in a wide variety of geometric and topological settings; see \cite{lachieze2019normal,lachieze2022stabilization,penrose2007gaussian}. Within the present context, these assumptions are tailored to ensure that the Tsallis entropy estimators considered in Section~\ref{sec:applications} allow exponential stabilization radius tails and sufficient moment control for the add-one costs, resulting in explicit normal approximation rates.
\end{remark}

\section{Main Results}
\label{sec:main_results}

In this section, we present the normal approximation bounds for stabilizing functionals of Poisson and binomial point processes. Here, \(N\) represents a standard normal random variable.

\subsection{Poisson input}

Suppose \(F\in L^2_{P_s}(\mathcal{X})\) is a square-integrable functional of a Poisson point process \(P_s\) having intensity \(\lambda=sQ\). Given a measurable subset \(A_x\subseteq\mathcal{X}\) depending on \(x\), the following are defined:
\begin{align*}
b_1(x,A_x) &:= \E\bigl|\Delta_x F - \Delta_x F(A_x)\bigr|^4,\\
b_2(x,A_x) &:= \E\bigl|\Delta_x F(A_x)\bigr|^4.
\end{align*}
For a measurable subset \( A_x \subset \mathcal{X} \), we define the localized add-one cost by
\[
\Delta_x F(A_x) := F\big((\eta \cap A_x)\cup \{x\}\big) - F(\eta \cap A_x).
\]

\begin{theorem} [Normal approximation under Poisson input]
\label{thm:normal_poisson}
 Let \(F\) be as above and assume that
\begin{equation}
\E[F^2] < \infty
\qquad\text{and}\qquad
\E\Bigl[\int_{\mathcal{X}} (\Delta_x F)^2 \,\lambda(dx)\Bigr] < \infty.
\end{equation}
Then there is an absolute constant \(C^*>0\) such as
\begin{equation}
d_K\Bigl(\frac{F-\E[F]}{\sqrt{\Var(F)}}, N\Bigr)
\le C^* \sum_{i=1}^6 \Gamma_i,
\end{equation}
The right-hand side gives a clear error bound in the Kolmogorov distance. 
Under suitable stabilisation and moment conditions, the \( \Gamma_i \) values 
tend to zero as the density parameter \( s \to \infty \); this indicates that the limit 
tends to zero and, consequently, asymptotic normality is achieved. 
Additionally, each \( \Gamma_i \) value 
is expressed by an explicit integral involving \( b_1 \), \( b_2 \), and combinations of first- and second-order growth terms; 
this expression is modified to fit the given cost operator framework, in the spirit of the \cite{lachieze2019normal} work.
\end{theorem}

\begin{remark}
Theorem~\ref{thm:normal_poisson} uses adaptable addition cost operators to localize increases without requiring a clear point decomposition. It can be seen as a modified improvement of the second-order Poincaré inequality in \cite{last2016normal} and the boundary conditions in \cite{lachieze2019normal} , which exhibit strong stabilization via an added  cost functional.
\end{remark}

\subsection{Binomial input}
Now, let us move on to functionals of binomial point processes. Suppose that for  \(n\ge2\), \(\xi_n\) is a binomial point process guided by \(Q\), and \(F_n \in L^2_{\xi_n}(\mathcal{X})\) is a strongly stabilizing functional satisfying Assumptions~\ref{Assumption 2.2} --~\ref{Assumption 2.4}.

\begin{theorem}[Normal approximation under binomial input]
\label{thm:normal_binomial}
Assume that \(F_n\) is strongly stabilizing with stabilization radius satisfying Assumptions~\ref{Assumption 2.2} --~\ref{Assumption 2.4}. In that case, there is a constant \(C'_0>0\), based only on the constants in those assumptions, such that for all \(n\ge2\),
\begin{equation}
    d_K\Bigl(\frac{F_n-\E[F_n]}{\sqrt{\Var(F_n)}}, N\Bigr)
\le C'_0\left(
\frac{\Theta_{K,n}^{1/2}}{\Var(F_n)}
+ \frac{\Theta_{K,n}}{\Var(F_n)^{3/2}}
+ \frac{\Theta_{K,n}+\Theta_{K,n}^{3/2}}{\Var(F_n)^2}
\right),
\end{equation}
where
\[
\Theta_{K,n}
:= n\int_{\mathcal{X}} \exp\Bigl(-\frac{C_4(p-4)}{4p}\, (d_n(x,K))^2\Bigr)\,Q(dx)
\]
and \(d_n(x,K):=\inf_{z\in K} n^{1/\beta} d(x,z)\). Here, \( \beta > 0 \) denotes the dimension parameter of the underlying space \( \mathcal{X} \).
\end{theorem}

\paragraph{Interpretation of \(\Theta_{K,s}\) and \(\Theta_{K,n}\).}
Quantities $\Theta_{K,s}$ and $\Theta_{K,n}$ represent the exponential decay of tail probabilities as a function of stabilization radii and insertion cost magnitudes. Informally, we measure how quickly the function has become insensitive to points away from the corresponding set \(K\). As these terms become uniformly bounded in variance, the resulting Kolmogorov bounds display the optimal rates \(s^{-1/2}\) or \(n^{-1/2}\).

\begin{remark}
The binomial limit in Theorem~\ref{thm:normal_binomial} corresponds to the Poisson case, although it has slightly different exponents because there is no exact binomial  analogue  of the second-order Poincaré inequality that exists in Poisson spaces; see \cite{lachieze2017new} for the corresponding Berry-Esseen bounds.
\end{remark}

\begin{corollary}[Optimal binomial convergence rate]
\label{cor:optimal_binomial}
We assume the setting of Theorem~\ref{thm:normal_binomial} and suppose  that \(C>0\) is a constant, so that
\begin{equation}
\sup_{n\ge1} \frac{\Theta_{K,n}}{\Var(F_n)} \le C.
\end{equation}
Then, there is a constant \(C''_0>0\), which depends only on \(\text{C}\) and the assumption constants, such that for all \(n\ge2\),
\[
d_K\Bigl(\frac{F_n-\E[F_n]}{\sqrt{\Var(F_n)}}, N\Bigr)
\le \frac{C''_0}{\sqrt{\Var(F_n)}}.
\]
Particularly, as \(Var(F_n)\) increases at least linearly with \(n\), this gives the normal approximation rate of $O(n^{-1/2})$. A similar result applies to the Poisson model under similar stabilization and moment conditions. Particularly, if \( \Var(F_s) \) grows linearly with \( s \), then Theorem~\ref{thm:normal_poisson}  shows that the convergence rate is of the order \( O(s^{-1/2}) \).
\end{corollary}

\medskip

Theorems~\ref{thm:normal_poisson} and~\ref{thm:normal_binomial} summarize the general Poisson and binomial limits in Table~\ref{tab:generic}.
\begin{table}[htbp]
\centering
\caption{Asymptotic normality and Kolmogorov estimates for general stabilizing processes under the assumptions of Section~\ref{sec:preliminaries}.}
\label{tab:generic}
\begin{tabular}{p{4.2cm}p{3.6cm}p{5cm}}
\hline
Functional & Input process & Rate at $d_K$ \\ \hline
$F$ & Poisson $P_s$ &
$C^{*}\, \displaystyle\sum_{i=1}^{6} \Gamma_i$ \\[0.3em]
$F_n$ & Binomial $\xi_n$ &
$C'_0\, \mathcal{R}(\Theta_{K,n},\Var(F_n))$ \\ \hline
\end{tabular}
\end{table}
where
\begin{equation}
\mathcal{R}(\Theta_{K,n},\Var(F_n)) =
\frac{\Theta_{K,n}^{1/2}}{\Var(F_n)}
+ \frac{\Theta_{K,n}}{\Var(F_n)^{3/2}}
+ \frac{\Theta_{K,n}+\Theta_{K,n}^{3/2}}{\Var(F_n)^2}.
\end{equation}
The constants $C^{*}$ and $C'_0$ are only dependent on structural assumptions and moment parameters, as in Theorem~\ref{thm:normal_binomial}.

\section{Applications and Simulation Strategy}
\label{sec:applications}

Now, we will explain the scope of Theorems~\ref{thm:normal_poisson} and \ref{thm:normal_binomial} using several functionals of practical interest, focusing particularly on Tsallis entropy estimators derived from nearest neighbor distances. We will also summarize a simulation strategy for verifying theoretical bounds in concrete settings.

\subsection{Tsallis entropy estimators based on nearest neighbors}

Let $f$ be a  density  over  \(\R^d\), For $\alpha \neq 1$,then the Tsallis entropy is given by
\begin{equation}
T_\alpha(f) = \frac{1}{1-\alpha}\left(\int_{\mathbb{R}^d} f(x)^\alpha\,dx - 1\right), \qquad \alpha \neq 1.
\end{equation}
which is recovered by Shannon entropy as \(\alpha\to1\) \cite{tsallis1988possible,cover2006elements}. Consider a Poisson process \(P_s\) on \(\mathcal{X}\subseteq\R^d\) with density \(sQ\). Given a constant integer \(k\ge1\), let \(\rho_k(x)\) denote the distance between \(x\in P_s\) and its \(k\)-th nearest neighbor in \(P_s\), and let \(\omega_d\) be the volume of the unit ball in \(\R^d\). The Tsallis entropy estimator is given by
\begin{equation}
F_s^{(\alpha)}(P_s)
= \frac{1}{1-\alpha}\left(
\frac{1}{|P_s|}\sum_{x\in P_s} \left(\frac{k}{s\,\omega_d \rho_k(x)^d}\right)^{1-\alpha}
- 1
\right),
\end{equation}
where \(|P_s|\) represents the number of points in \(P_s\). Similarly, the corresponding estimator \(F_n^{(\alpha)}(\xi_n)\) for binomial input is defined, but using \(n\) instead of \(s\) and \(\xi_n\) instead of \(P_s\).

\begin{theorem}[Normal approximation for Tsallis entropy estimators]
\label{thm:tsallis_entropy}
Assume that Assumptions ~\ref{Assumption 2.1} -- ~\ref{Assumption 2.4} are valid for \((\mathcal{X},d, Q)\), and let the Tsallis entropy estimator \(F_s^{(q)}(P_s)\) (respectively, \(F_n^{(\alpha)}(\xi_n)\)) describe a strongly stabilizing functional having stabilization radii and add-one costs that follow the exponential tail and uniform moment requirement. Then there are constants \(C_P,C_B>0\), determined only by the assumptions and by \(\alpha,k,d\), where for all \(s\ge1\) and \(n\ge2\),
\begin{equation}
d_K\Bigl(
\frac{F_s^{(\alpha)}(P_s) - \E[F_s^{(\alpha)}(P_s)]}{\sqrt{\Var(F_s^{(\alpha)}(P_s))}},
N
\Bigr) \le \frac{C_P}{\sqrt{s}},
\end{equation}
and
\[
d_K\Bigl(
\frac{F_n^{(\alpha)}(\xi_n) - \E[F_n^{(\alpha)}(\xi_n)]}{\sqrt{\Var(F_n^{(\alpha)}(\xi_n))}},
N
\Bigr) \le \frac{C_B}{\sqrt{n}}.
\]
\end{theorem}

\begin{remark}
Theorem~\ref{thm:tsallis_entropy} establishes a unified asymptotic normality consequence for a natural class of nearest-neighbor-based estimators of Tsallis entropy and supplements previous results on the consistency and mean-square approximation of such estimators in the context of goodness-of-fit \cite{cadirci2022ggd,cadirci2025tsallisGOF}. Under similar stabilization assumptions, the \(s^{-1/2}\) and \(n^{-1/2}\) rates correspond to the behavior exhibited by Shannon and Rényi entropy estimators and are thus classical central limit scaling classical rates. Thus, the current finding presents a  framework similar to the Rényi entropy estimation proposed in \cite{cadirci2025renyi}, but it is supported by stabilization techniques and Poisson/binomial estimation methods.
\end{remark}

\subsection{Weighted \(k\)-NN Shannon entropy}
To complete, let us briefly recall the weighted \(k\)-nearest neighbor estimator of Shannon entropy recommended in \cite{berrett2019efficient}. Assume that \(X_1,\dots,X_n\) are independent and identically distributed with density \(q\) on \(\R^d\). Then the Shannon entropy is
\begin{equation}
H(q) = -\int_{\R^d} q(x)\log q(x)\,dx.
\end{equation}
Let \(\rho_{j,i}\) denote the distance from \(X_i\) to its \(j\)-th nearest neighbor in the set \(\{X_1,\dots,X_n\}\setminus\{X_i\}\), Suppose that \(V_d\) is the volume of the unit ball in \(\R^d\) and \(\Psi\) is the digamma function. We obtain the estimator for weights \(\{w_j\}_{j=1}^k\) satisfying the condition \(\sum_{j=1}^k w_j = 1\)
\begin{equation}
F_n^{\mathrm{SE}}(\xi_n)
= \frac{1}{n}\sum_{i=1}^n \sum_{j=1}^k w_j
\log\left(\frac{(n-1)V_d \rho_{j,i}^d}{e^{\Psi(j)}}\right)
\end{equation}
has appropriate bias and variance properties under mild regularity conditions.

\begin{theorem}[Normal approximation for weighted \(k\)-NN Shannon estimator]
\label{thm:shannon_entropy}
Given ~\ref{Assumption 2.2} -- ~\ref{Assumption 2.4} and the conditions of regularity from \cite{berrett2019efficient} about \(q\) and the weights \(\{w_j\}\), let there be constants \(C_0,\tau>0\) that, based only on these conditions, are such that for all sufficiently large \(n\),
\begin{equation}
d_K\Bigl(
\frac{F_n^{\mathrm{SE}}(\xi_n) - H(q)}{\sqrt{\Var(F_n^{\mathrm{SE}}(\xi_n))}},
N
\Bigr)
\le C_0 \left(\frac{k}{n}\right)^\tau.
\end{equation}
\end{theorem}

\subsection{Geometric functionals: Euler characteristic and MST}
For a geometric example, we consider a binomial point process \(\xi_n\) on \([0,1]^d\) with density \(q\) being zero and bounded away from infinity. Denote the Vietoris-Rips or Čech complex with scale \(r>0\) by \(K_r\) and write the Euler characteristic as \(\chi(K)\). We define
\begin{equation}
F_n^{\mathrm{EC}}(P_n) = \chi(K_r(n^{1/d} P_n)),
\qquad
F_n^{\mathrm{EC}}(\xi_n) = \chi(K_r(n^{1/d} \xi_n)).
\end{equation}

\begin{theorem}[Normal approach for the Euler characteristic]
\label{thm:euler_characteristic}
Assume that  ~\ref{Assumption 2.2} -- ~\ref{Assumption 2.4} apply to the corresponding functionals of \(P_n\) and \(\xi_n\). Then there exists a constant \(C_{\mathrm{EC}}>0\) with the following property:
\begin{equation}    
d_K\Bigl(
\frac{F_n^{\mathrm{EC}}(\eta_n) - \E[F_n^{\mathrm{EC}}(\eta_n)]}{\sqrt{\Var(F_n^{\mathrm{EC}}(\eta_n))}},
N
\Bigr)
\le \frac{C_{\mathrm{EC}}}{\sqrt{n}},
\end{equation}
where \(\eta_n\) denotes either \(P_n\) or \(\xi_n\).
\end{theorem}

To find the normal approach rates for minimal spanning trees, let \(V\subset\R^d\) be a finite set of points and let \(M(V)\) represent the total length of the minimal spanning tree on \(V\). Given that \(V\) is generated by a Poisson process with density \(nQ\), stabilization arguments and cost function techniques provide the following normal approach rates.
\[
d_K\Bigl(
\frac{F_n^{\mathrm{MST}} - \E[F_n^{\mathrm{MST}}]}{\sqrt{\Var(F_n^{\mathrm{MST}})}},
N
\Bigr)
\le
\begin{cases}
C_{\mathrm{MST}}\, n^{-\gamma_1}, & d=2,\\[0.4em]
C_{\mathrm{MST}}\, (\log n)^{-\gamma_2}, & d\ge3,
\end{cases}
\]
for appropriate constants \(C_{\mathrm{MST}},\gamma_1,\gamma_2>0\) based on the basic distribution and dimension; see \cite{kesten1996central,chatterjee2017minimal} for corresponding results.

\medskip

We highlight the different rates achieved for the main functionals discussed in Theorems~\ref{thm:normal_poisson}, \ref{thm:normal_binomial}, \ref{thm:tsallis_entropy}, \ref{thm:shannon_entropy}, and Theorem~\ref{thm:euler_characteristic}. We have summarized the corresponding Kolmogorov estimates in Table~\ref{tab:rates}.

\begin{table}[htbp]
\centering
\caption{The asymptotic normality and Kolmogorov rates for a set of entropy-related and geometric functionals under the stabilization conditions of Section~\ref{sec:preliminaries}.}
\label{tab:rates}
\begin{tabular}{p{4cm}p{3.2cm}p{3.2cm}p{3.2cm}}
\hline
Functional & Input process & Target quantity & Rate in $d_K$ \\ \hline
Tsallis $k$-NN estimator $F_s^{(\alpha)}$ & Poisson $P_s$ & Tsallis entropy & $C_{P}\, s^{-1/2}$ \\
Tsallis $k$-NN estimator $F_n^{(\alpha)}$ & Binomial $\xi_n$ & Tsallis entropy & $C_{B}\, n^{-1/2}$ \\
Weighted $k$-NN estimator $F_n^{\mathrm{SE}}$ & Binomial $\xi_n$ & Shannon entropy & $C_0 (k/n)^{\tau}$ \\
Euler characteristic $F_n^{\mathrm{EC}}$ & Poisson or binomial & Euler characteristic & $C_{\mathrm{EC}}\, n^{-1/2}$ \\
MST total length $F_n^{\mathrm{MST}}$ & Poisson $P_n$ & MST length & $C_{\mathrm{MST}} n^{-\gamma_1}$ or $C_{\mathrm{MST}} (\log n)^{-\gamma_2}$ \\ \hline
\end{tabular}
\end{table}

\subsection{Outline of a simulation study}
We examine both Poisson and binomial point process models. 
For Poisson processes, point configurations are generated from a homogeneous Poisson process with density \( s \) in a bounded region of \( \mathbb{R}^d \). 
For the binomial case, we produce samples of size \( n \) which are independent of the specified distributions and identically distributed.

We focus on Tsallis entropy functionals computed using \(k\)-nearest-neighbor methods to examine the behavior of estimators. 
Additionally, to demonstrate the flexibility of the framework, we include weighted nearest-neighbor estimators related to Shannon entropy.

We determine the normalized estimator for each configuration and approximate its distribution using Monte Carlo replications. 
The Kolmogorov distance between the empirical distribution of the normalized estimator and the standard normal distribution is then used to measure the accuracy of the normal approximation.

In order to investigate both well-defined and poorly specified scenarios, we generate data from generalized Gaussian distributions with varying shape parameters and from alternative distributions to assess robustness.
This setup enables us to evaluate the convergence behavior and sharpness of the theoretical bounds derived in Section~\ref{sec:preliminaries}.
While our primary focus is theoretical, it is natural to evaluate the sharpness of the derived bounds using Monte Carlo experiments. We briefly summarize a general simulation design consistent with the entropy-based goodness-of-fit frameworks in \cite{cadirci2022ggd,cadirci2025tsallisGOF,cadirci2025renyi}.
In a typical setup, dimensions $d\in\{1,2,5\}$, sample sizes $n\in\{200,500,1000,2000\}$, and various values of the Tsallis or Rényi parameter $\alpha\in\{0.8,1.0,1.2\}$ are considered. Under the null hypothesis, simulations can be performed from multivariate generalized Gaussian distributions with different shape parameters. For each configuration, Tsallis, Shannon, and Rényi entropy estimators can be calculated, their empirical means and variances can be estimated, and the standardized estimator can be compared to the Gaussian benchmark across Monte Carlo replications to approximate the Kolmogorov distance to the normal distribution.

Main empirical variables of interest are:
\begin{itemize}
\item Bias and mean squared error of the Tsallis and Rényi entropy estimators for various $n,d,\alpha$,
\item Empirical variance and its scaled value in $n$,
\item Empirical Kolmogorov distance between the normalized estimators and a normal distribution,
\item Empirical size and power of Tsallis- and Rényi-based goodness-of-fit tests developed from the estimators.
\end{itemize}

Figure~\ref{fig:variance_loglog} displays the empirical variance of the Tsallis and weighted Shannon $k$-NN entropy estimators as a function of sample size $n$ for several comparison distributions and sizes on a double logarithmic scale. In all scenarios examined, the variance is approximately linearly decreasing with $\log n$. It has a slope close to $-1$, consistent with the theoretical estimate $\Var(F_n)=O(n^{-1})$ suggested by the central limit theorems in Section~\ref{sec:main_results}. The alignment with the reference $O(n^{-1})$ line is particularly evident for medium and large sample sizes. It is robust with respect to changes in the underlying distribution (light-tailed generalized Gaussian, heavy-tailed generalized Gaussian, and Student-$t$) and dimension ($d=1$ and $d=5$). The Tsallis estimator exhibits slightly higher variance than the Shannon estimator in heavy-tailed and high-dimensional settings, but both estimators  have the same asymptotic slope, suggesting that the stabilization-based normal approximations capture the dominant scaling behaviour of the fluctuation magnitude. These empirical findings support the optimality of the $n^{-1/2}$ rates at our Kolmogorov limits and show that theoretical results remain informative even at moderately high dimensions and under significant deviations from the Gaussian distribution.

\begin{figure}[htbp]
\centering
\includegraphics[width=\textwidth]{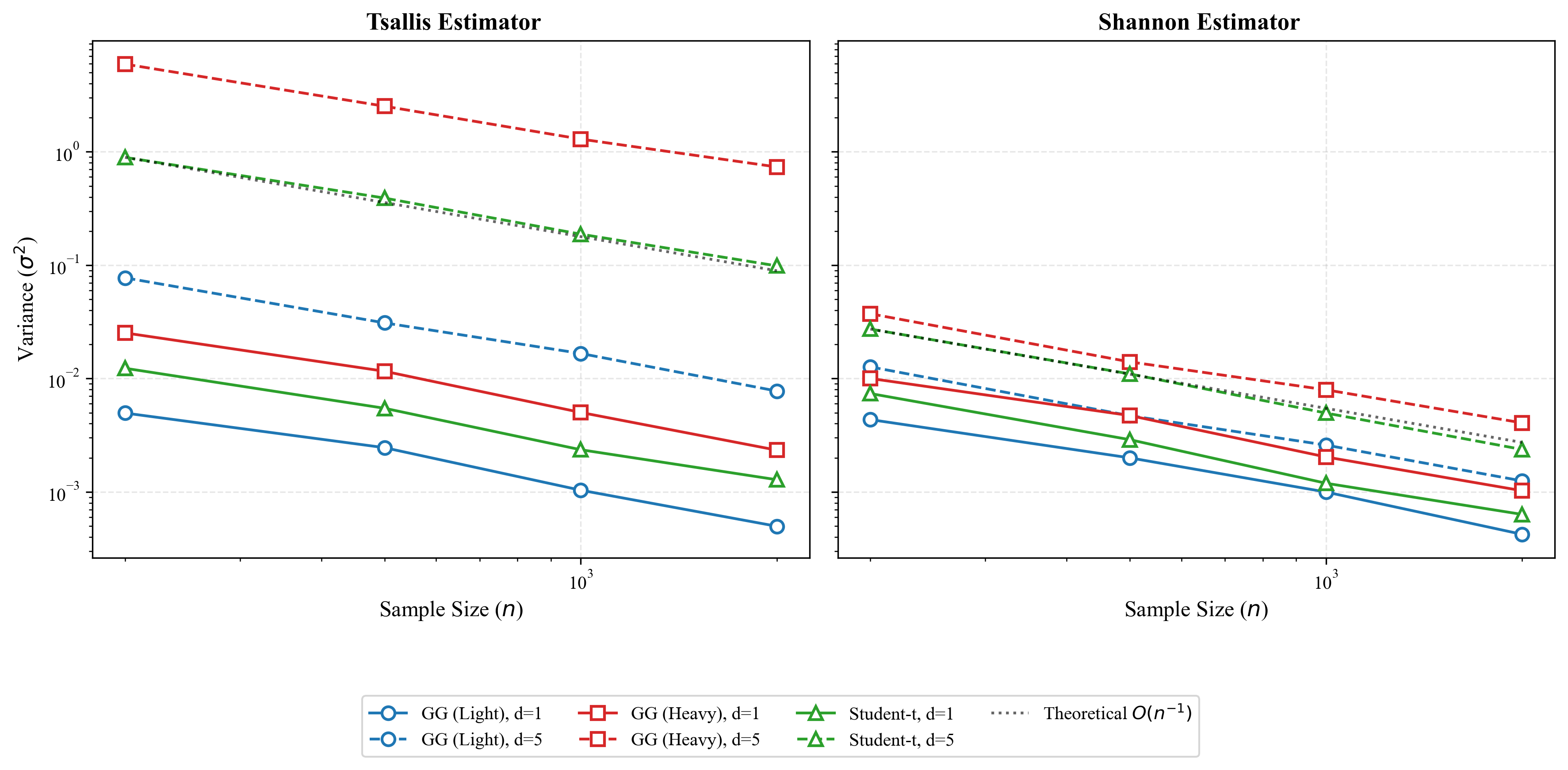}
\caption{Tsallis (left) and weighted Shannon (right) $k$-NN entropy estimators, showing the empirical variance as a function of sample size $n$ on a log-log scale for various distributions and dimensions. The grey dotted line represents the reference $O(n^{-1})$ variance decrease implied by the central limit theorems.}
\label{fig:variance_loglog}
\end{figure}
Figure~\ref{fig:asymptotic_normality} supplies a qualitative evaluation of the asymptotic normality of the Tsallis and weighted Shannon $k$-NN entropy estimators. We generate $n\in\{200,500,1000,2000\}$ observations for a light-tailed bivariate generalized Gaussian distribution ($d=2$, shape parameter $\beta=2.5$) and computed over 300 Monte Carlo replications, which were found to provide stable estimates while keeping computational costs moderate. Monte Carlo replications. We use the following formula to standardize the resulting estimates $\hat{H}$ for each $n$:
\begin{equation}
Z_n = \frac{\hat{H} - \E[\hat{H}]}{\sqrt{\Var(\hat{H})}},
\end{equation}
ensuring that any deviation from the standard normal distribution is  reflected in the shape of the standardized sample.
\begin{figure}[htbp]
\centering
\includegraphics[width=0.85\textwidth]{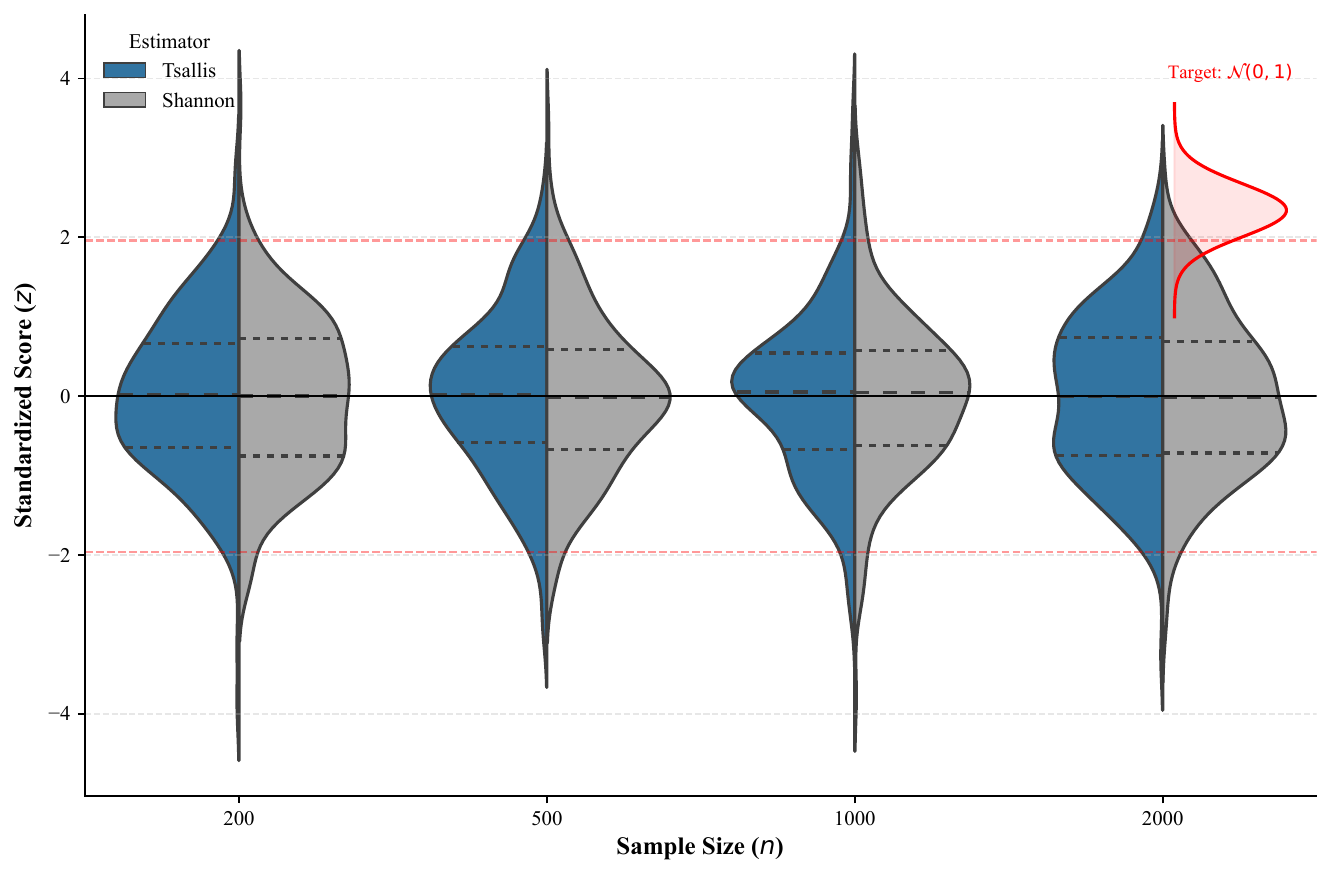}
\caption{Violin plots for standardized Tsallis and weighted Shannon \(k\)-NN entropy estimators for the bivariate generalized Gaussian distribution (\(d = 2\), \(\beta = 2.5\)) with sample sizes \(n \in \{200, 500, 1000, 2000\}\). For each \(n\), standardized scores \(Z_n\) were obtained from 200 Monte Carlo replications. The horizontal dashed lines indicate the $\pm 1.96$ quantiles of the standard normal distribution. The consistency across sample sizes illustrates the rapid convergence to normality anticipated by the \(O(n^{-1/2})\) rate.}
\label{fig:asymptotic_normality}
\end{figure}

The violin plots in Figure~\ref{fig:asymptotic_normality} show how the empirical distributions of the standardized Tsallis and Shannon estimators change as the sample size increases. The fact that the distribution’s mean is at zero and the reference quantile alignment indicate that the normal approximation holds for every \( n \) value examined by the researchers. The $n=200$ sample shows that both estimators deviate from Gaussianity due to their asymmetrical distributions and  additional tail weight. The distribution shapes develop symmetric patterns while their distribution mass begins to concentrate at the origin. The Tsallis estimator displays a dispersion pattern that is very similar to that of the Shannon estimator at all sample sizes considered, indicating that the use of a generalized entropy index (here $\alpha=1.2$) does not lead to slower convergence or markedly heavier tails. The \( O(n^{-1/2}) \) convergence rate with theoretical variance scaling indicates that higher \( n \) values lead to a decrease in variability. According to empirical results supporting the theoretical findings, the research findings indicate that stabilization-based normal approximations accurately characterize the distribution behavior of estimators. The visual diagnostics confirm the quantitative Kolmogorov distance results, which match the central limit theorems established in Section~\ref{sec:main_results} because both estimators reach asymptotic normality under the stabilization and moment assumptions established in this study.

\section{Proofs of Main Results}
\label{sec:proofs}

The section provides proofs of all main normal approximation results documented in Section~\ref{sec:main_results}. The study demonstrates how adaptive add-one cost decompositions interact with second-order Poincaré-type inequalities.

\subsection{Poisson case}

Theorem~\ref{thm:normal_poisson} proof requires two components, which include the second-order Poincaré inequality for Poisson functionals proved by \cite{last2016normal} together with an adaptive decomposition of the increment operators. For convenience, we recall a simplified version of the main inequality.

\begin{theorem}[Last--Peccati--Schulte {\cite{last2016normal}}]
\label{thm:last_peccati_schulte}
Suppose \(F\) is a square-integrable function of the Poisson point process \(P(\lambda)\) with intensity measure \(\lambda\), satisfying
\begin{equation}
\E[F^2] < \infty
\qquad\text{and}\qquad
\E\Bigl[\int_{\mathcal{X}} (\Delta_x F)^2\,\lambda(dx)\Bigr] < \infty.
\end{equation}
Then
\[
d_K\Bigl(\frac{F-\E[F]}{\sqrt{\Var(F)}}, N\Bigr)
\le \sum_{i=1}^6 \gamma_i,
\]
where \(\gamma_1,\dots,\gamma_6\) denote explicit integrals involving the first and second derivatives of \(F\).
\end{theorem}

\textbf{Proof of Theorem~\ref{thm:normal_poisson}.}
We derive first-order increments using the following equation:
\[
\Delta_x F = \bigl(\Delta_x F - \Delta_x F(A_x)\bigr) + \Delta_x F(A_x),
\]
and applying the same method, we decompose second-order increments \(\Delta_{x,y}F\) into three parts that characterize different localization effects. The first moment quantities of \(b_1(x,A_x)\) and \(b_2(x,A_x)\), together with the Hölder inequality, provide upper bounds for all \(\gamma_i\) via the following integral combinations
\[
\int_{\mathcal{X}} b_1(x, A_x)^{3/4}\,\lambda(dx),
\qquad
\int_{\mathcal{X}} b_2(x, A_x)^{3/4}\,\lambda(dx).
\]
Moreover, other components involving second-order increments provide upper bounds for all \(\gamma_i\). The combined bounds yield the desired result, since we have combined the numerical constants into \(C^*\) and defined \(\Gamma_i\) via the relevant integral formulas. \hfill\(\square\)

\subsection{Binomial case}

Theorem~\ref{thm:last_peccati_schulte} does not have its exact equivalent for the binomial setting yet \cite{lachieze2017new} proved related inequalities. Theorem~\ref{thm:normal_binomial} results from applying their proof methods to our cost-operator framework. The process for this method follows the same steps as the Poisson method, except that discrete sums replace integrals and suitable combinatorial factors are used.

\medskip

\textbf{Proof of Theorem~\ref{thm:normal_binomial}.} We decompose \(\Delta_x F_n\) and \(\Delta_{x,y}F_n\) using adaptive subsets \(A_x\) as in the Poisson case and apply the binomial Berry--Esseen bounds developed in \cite{lachieze2017new}. The exponential tail and moment assumptions control the contributions from large increments and yield the stated dependence on \(\Theta_{K,n}\) and \(\Var(F_n)\). The details which follow the structure from \cite{lachieze2017new} have been omitted because of our need for shorter content. \hfill\(\square\)

\subsection{Proof of Corollary~\ref{cor:optimal_binomial}}

The corollary follows by choosing \(A_x=\mathcal{X}\), which simplifies the decomposition so that \(b_1\equiv0\) and the main contribution comes from \(b_2\). The assumptions imply that \(\Theta_{K,n}\) stays under uniform control when compared to \(\Var(F_n)\) and Theorem~\ref{thm:normal_binomial} establishes a bound which reduces to the value of \(\Var(F_n)^{-1/2}\). The details are similar to those of the Poisson case examined in \cite{lachieze2019normal}, and therefore we will not include them.

\section{Conclusion}
\label{sec:conc}

The research establishes general stabilization-based processes that researchers can use to prove central limit theorems, providing specific Kolmogorov \rso{and Wasserstein} convergence rates for all functions derived from Poisson and binomial point processes. This method provides first- and second-order growth control through the systematic application of adaptable addition cost operators, eliminating the need for detailed stabilization radius calculations and explicit point-function decompositions.

The framework provides optimal-order normal approximation results for nearest-neighbor Tsallis entropy estimators based on natural regularity and tail and moment conditions applied to the underlying metric measure space. The same methodology recovers and extends existing normal approximation results for weighted $k$-NN Shannon entropy estimators, Euler characteristics of random geometric complexes, and minimal spanning tree functionals, highlighting the unifying role of stabilization techniques in geometric probability and  information-theoretic estimation.

The results presented here suggest several directions for further research. This work can be extended by testing Tsallis and related entropy estimators in non-Euclidean spaces, such as manifolds and graphs, as these spaces possess unique geometric properties that affect nearest-neighbor relations and stabilization methods. It can be adapted to select the neighbor parameter $k$. It will investigate how data-driven approaches perform in high-dimensional settings, examining their impact on bias and variance during stabilization.The current normal approach should be better combined with entropy-based goodness-of-fit tests and dependency modelling to develop tests and estimators that exhibit strong asymptotic performance in complex, high-dimensional contexts while preserving reliable results.






\bibliographystyle{plain}
\bibliography{bibliography}

\end{document}